\documentclass{amsart}

\begin{document}

\markboth{Mancho Manev}
{Manifolds with Almost Contact 3-Structure and Metrics of HN-Type}

\renewcommand\theenumi{(\roman{enumi})}
\renewcommand\labelenumi{{\theenumi}}
%
%
 \newtheorem{thm}{Theorem}[section]
 \newtheorem{cor}[thm]{Corollary}
 \newtheorem{lem}[thm]{Lemma}
 \newtheorem{prop}[thm]{Proposition}
 \theoremstyle{definition}
 \newtheorem{defn}[thm]{Definition}
 \theoremstyle{remark}
 \newtheorem{rem}[thm]{Remark}
 \newtheorem*{ex}{Example}
 \numberwithin{equation}{section}

\title[Manifolds with Almost Contact 3-Structure and Metrics of HN-Type]
{Manifolds with Almost Contact 3-Structure \\ and Metrics of Hermitian-Norden Type}

\author[M. Manev]{Mancho Manev} 


\dedicatory{Dedicated to Professor Georgi Ganchev on his seventieth birthday}


%

\address[1]{
Department of Algebra and Geometry \endgraf
Faculty of Mathematics and Informatics \endgraf
Plovdiv University \endgraf
236 Bulgaria Blvd  \endgraf
Plovdiv 4027 \endgraf
Bulgaria}
\email{mmanev@uni-plovdiv.bg}

\address[2]{
Department of Medical Informatics \endgraf
Biostatistics and Electronic Education  \endgraf
Faculty of Public Health  \endgraf
Medical University of Plovdiv \endgraf
15A Vasil Aprilov Blvd \endgraf
Plovdiv 4002 \endgraf
Bulgaria}
\email{mmanev@meduniversity-plovdiv.bg}


\newcommand{\MM}{\mathcal{M}}
\newcommand{\R}{\mathbb{R}}
\newcommand{\X}{\mathfrak{X}}
\newcommand{\f}{\varphi}
\newcommand{\ep}{\varepsilon}
\newcommand{\ea}{\varepsilon_\alpha}
\newcommand{\eb}{\varepsilon_\beta}
\newcommand{\eg}{\varepsilon_\gamma}
\newcommand{\n}{\nabla}
\newcommand{\W}{\mathcal{W}}
\newcommand{\al}{\alpha}
\newcommand{\bt}{\beta}
\newcommand{\gm}{\gamma}
\newcommand{\lm}{\lambda}
\newcommand{\ta}{\theta}
\newcommand{\fa}{\varphi_{\alpha}}
\newcommand{\xia}{\xi_{\alpha}}
\newcommand{\etaa}{\eta_{\alpha}}
\newcommand{\om}{\omega}
\newcommand{\D}{\mathrm{d}}
\newcommand{\LL}{\mathcal{L}}
\newcommand{\LLL}{\mathfrak{L}}
\newcommand{\SSS}{\mathcal{S}}
\newcommand{\C}{\mathbb{C}}
\newcommand{\F}{\mathcal{F}}
\newcommand{\ddt}{\textstyle\frac{\D}{\D t}}
\newcommand{\Ja}{J_{\alpha}}
\newcommand{\G}{\mathcal{G}}
\newcommand{\V}{\mathcal{V}}
\newcommand{\pd}{\partial}
\newcommand{\ddx}{\frac{\pd}{\pd x^i}}
\newcommand{\ddy}{\frac{\pd}{\pd y^i}}
\newcommand{\ddu}{\frac{\pd}{\pd u^i}}
\newcommand{\ddv}{\frac{\pd}{\pd v^i}}
\newcommand{\dda}{\frac{\pd}{\pd a}}
\newcommand{\ddb}{\frac{\pd}{\pd b}}
\newcommand{\ddc}{\frac{\pd}{\pd c}}
\newcommand{\Lfa}{(L,\allowbreak\f_{\al},\allowbreak\xi_{\al},\allowbreak\eta_{\al},g)}

\newcommand{\thmref}[1]{The\-o\-rem~\ref{#1}}
\newcommand{\propref}[1]{Pro\-po\-si\-ti\-on~\ref{#1}}


\begin{abstract}
It is introduced a differentiable manifold with almost contact 3-struc\-ture which consists of an almost contact metric structure and two almost contact B-metric structures. 
The product of this manifold and a real line is an almost hypercomplex manifold with Hermitian-Norden metrics. 
It is proven that the in\-tro\-duced manifold of cosymplectic type is flat. Some examples of the studied man\-i\-folds are given.
\end{abstract}

\maketitle

\section{Introduction}\label{intro}

The notion of \emph{almost contact 3-structure} is introduced by Y.Y. Kuo in \cite{Kuo} and independently under the name \emph{almost coquaternion structure} by C. Udri\-\c{s}te in \cite{Udr}. Later, it is studied by several authors. It is well known that the product of a man\-i\-fold with almost contact 3-structure and a real line admits an \emph{almost hypercomplex structure} (also known as an \emph{almost quaternion structure}) (cf. \cite{Kuo,AlMa}).

All authors previously considered the case when there exists a Riemannian metric compatible with each of the three structures in the given almost contact 3-structure. Then the object of study is the so-called \emph{almost contact metric 3-structure}.

In the present work, on a $(4n+3)$-dimensional manifold with almost contact 3-structure we introduce a pseudo-Riemannian metric which has another kind of compatibility with the triad of almost contact structures. The product of this manifold of new type and a real line is a $(4n+4)$-dimensional manifold which admits an almost hypercomplex structure $(J_1,J_2,J_3)$ and a Hermitian-Norden metric.
In our case, one of the almost complex structures (resp., the other two almost complex structures) of $(J_1,J_2,J_3)$ acts as an
isometry (resp., act as anti-isometries) with respect to the pseudo-Riemannian metric $G$ of neutral signature in each tangent fibre.
The metric $G$ is Hermitian with respect to the former structure of
$(J_1,J_2,J_3)$ and $G$ is a Norden metric (also known as an anti-Hermitian metric) regarding the latter structures of $(J_1,J_2,J_3)$. This structure is called an \emph{almost hypercomplex HN-metric structure} (HN stands for Hermitian-Norden) and it is studied in \cite{GriManDim12,GriMan24,Man28,ManGri32}, etc.

The goal of this paper is to launch a study of the manifolds with almost contact 3-structure and metrics of an HN-type.

\subsection*{Convention}
Let $\MM$ be an almost contact manifold and $\MM\times\R$ be the corresponding almost complex manifold.
\begin{enumerate}
  \item We shall use $x$, $y$, $z$, $\dots$ to denote smooth vector fields on $\MM$, i.e. $x, y, z\in \X(\MM)$, or vectors in the tangent space $T_p\MM$ at $p\in \MM$;
  \item We shall use $X$, $Y$, $Z$, $\dots$ to denote smooth vector fields on $\MM\times\R$ or tangent vectors in $T_{\bar{p}}(\MM\times\R)$ at $\bar{p}\in \MM\times\R$.
\end{enumerate}

\subsection*{Acknowledgments}
This work was partially supported by project NI15-FMI-004 of the
Scientific Research Fund at the University of Plovdiv.
The author wishes to thank Professor Stefan Ivanov for valuable discussions about the present paper.

\section{Almost contact metric manifolds}

Let $\MM$ be an odd-dimensional smooth manifold which is compatible
with an almost contact structure $(\f_{1},\xi_{1},\eta_{1})$, i.e. $\f_1$ is an endomorphism
of the tangent bundle, $\xi_1$ is a Reeb vector field and $\eta_1$ is its dual contact 1-form
satisfying the following identities:
\begin{equation}\label{str-al=1}
\begin{array}{c}
(\f_1)^2 = -I + \xi_1\otimes\eta_1,\quad
\f_1\xi_1 = o,\quad
\eta_1\circ\f_1=0,\quad \eta_1(\xi_1)=1,
\end{array}
\end{equation}
where $I$ is the identity in the algebra $\X(\MM)$ and $o$ is the zero element of $\X(\MM)$.
Moreover, let $g$ be a pseudo-Riemannian metric on $\MM$ which is compatible
with $(\f_{1},\xi_{1},\eta_{1})$ as follows:
\begin{equation}\label{A-met}
\begin{array}{c}
  g(\xi_{1},\xi_{1})=-\ep, \quad \eta_{1}(x)=-\ep g(\xi_{1},x),\\
  g(\f_{1} x,\f_{1} y)=g(x,y)+\ep\eta_{1}(x)\eta_{1}(y),
\end{array}
\end{equation}
where  $\ep=+1$ or $\ep=-1$. Then $(\f_{1},\xi_{1},\eta_{1},g)$ is called an \emph{almost contact metric structure} on $\MM$.

Usually, one can assume that $\ep=-1$ in \eqref{A-met} without loss of generality. This is conditioned since if we put $\bar\f_{1}=\f_{1}$, $\bar\xi_{1}=-\xi_{1}$, $\bar\eta_{1}=-\eta_{1}$, $\bar g=-g$, then $(\bar\f_{1},\bar\xi_{1},\bar\eta_{1},\bar g)$ for $\bar\ep=-\ep$ is also an almost contact metric structure on $\MM$ \cite{Taka}.
Here, we pay attention of the case of $\ep=+1$ which is in relation with our topic. We call the corresponding $(\f_{1},\xi_{1},\eta_{1},g)$-structure 
an \emph{almost contact metric structure} and $g$ --- a \emph{compatible metric} on $\MM$.

Since $g$ is a Hermitian metric with respect to the almost complex structure $\f_1\vert_{H_1}$ on the contact distribution $H_1=\ker(\eta_1)$, any metric with properties \eqref{A-met} can be considered as an odd-dimensional counterpart of the corresponding pseudo-Riemannian Hermitian metric, or this compatible metric is a pseudo-Riemannian metric of Hermitian type on an odd-dimensional differentiable manifold.

A classification of the almost contact metric manifolds is given in \cite{AlGa}. There, it is considered the vector space of the tensors of type $(0,3)$ defined by $F_1(x,y,z)=g\left(\left(\n_x\f_1\right)y,z\right)$, where $\n$ is the Levi-Civita connection generated by $g$.
They have the following basic properties
\begin{equation}\label{F1-prop}
\begin{array}{l}
  F_1(x,y,z)=- F_1(x,z,y)\\
  \phantom{F_1(x,y,z)}
  =- F_1(x,\f_1y,\f_1z)+F_1(x,\xi_1,z)\,\eta_1(y)\\
  \phantom{F_1(x,y,z)=- F_1(x,\f_1y,\f_1z)}
  +F_1(x,y,\xi_1)\,\eta_1(z).
\end{array}
\end{equation}
This vector space is decomposed in 12 orthogonal and invariant subspaces with respect to the action of the structural group $U\left(m\right)\times 1$, $m=\frac12(\dim\MM-1)$, where $U$ is the unitary group. In such a way, it is obtained a classification in 12 basic classes $\W_i$ $(i=1,2,\dots,12)$ with respect to $F_1$. Bearing in mind the above remarks we can use the same classification for almost contact metric manifolds.

The class $\W_0$ of cosymplectic  manifolds is contained in any other class.
There are $2^{12}$ classes at all. Some of these classes are discovered before the complete classification and they are known by special names. For example,
$\W_2\oplus\W_4$ is the class of quasi-Sasakian manifolds,
$\W_3\oplus\W_5$ is the class of quasi-Kenmotsu manifolds,
$\W_2\oplus\W_3$ is the class of trans-Sasakian manifolds of type $(\al,\bt)$
 and so on.

Bearing in mind \eqref{A-met}, we establish that the covariant derivatives of the fundamental tensors with respect to $\n$ are related by
\begin{equation}\label{Fetaxi1}
  F_1(x,\f_1y,\xi_1)=- \left(\n_x\eta_1\right)(y)=g\left(\n_x\xi_1,y\right).
\end{equation}

\section{Almost contact B-metric manifolds}

Let $\MM$ be equipped with another almost contact structure $(\f_2,\xi_2,\eta_2)$. The metric $g$ is called a \emph{B-metric} on this almost contact manifold if it obey the following relations:
\begin{equation}\label{B-met}
\begin{array}{c}
  g(\xi_{2},\xi_{2})=1, \quad \eta_{2}(x)=g(\xi_{2},x),\\
  g(\f_{2} x,\f_{2} y)=-g(x,y)+\eta_{2}(x)\eta_{2}(y).
\end{array}
\end{equation}
In this case $(\f_{2},\xi_{2},\eta_{2},g)$ is called an \emph{almost contact B-metric structure} on $\MM$ \cite{GaMiGr}.

Since $g$ is a Norden metric 
with respect to the almost complex structure $\f_2$ on the contact distribution $H_2=\ker(\eta_2)$, any B-metric can be considered as an odd-dimensional counterpart of the corresponding Norden metric, or the B-metric is a pseudo-Riemannian metric of Norden type on an odd-dimensional differentiable manifold.

A classification of the almost contact B-metric manifolds is given in \cite{GaMiGr}. There, it is considered the vector space of the fundamental tensors of type $(0,3)$ defined by $F_2(x,y,z)=g\left(\left(\n_x\f_2\right)y,z\right)$ and having the properties
\begin{equation}\label{F2-prop}
\begin{array}{l}
  F_2(x,y,z)= F_2(x,z,y)\\
  \phantom{F_2(x,y,z)}
  = F_2(x,\f_2y,\f_2z)+F_2(x,\xi_2,z)\,\eta_2(y)\\
  \phantom{F_2(x,y,z)= F_2(x,\f_2y,\f_2z)}
    +F_2(x,y,\xi_2)\,\eta_2(z).
\end{array}
\end{equation}
This vector space is decomposed in 11 orthogonal and invariant subspaces with respect to the action of the structural group $(GL(m;\C)\cap O(m,m))\times 1$, $m=\frac12(\dim\MM-1)$, where $O(m,m)$ is the pseudo-orthogonal group of neutral signature. Thus, it is obtained a classification in 11 basic classes $\F_i$ $(i=1,2,\dots,11)$ with respect to $F_2$.

The intersection of the basic classes is the special class $\F_0$
determined by the condition $F_2=0$. This is the class of \emph{cosymplectic B-metric manifolds}.

From \eqref{A-met}, we obtain the following relations between the covariant derivatives of the fundamental tensors with respect to $\n$:
\begin{equation}\label{Fetaxi2}
  F_2(x,\f_2y,\xi_2)= \left(\n_x\eta_2\right)(y)=g\left(\n_x\xi_2,y\right).
\end{equation}

\section{Manifolds with almost contact HN-metric 3-structure}\label{sec:1}

Let $(\MM,\f_{\al},\xi_{\al},\eta_{\al})$ $(\al=1,2,3)$ be a manifold
with an almost contact 3-struc\-ture, 
i.e. $\MM$
is a $(4n+3)$-dimensional differentiable manifold with three almost
contact structures $(\f_{\al},\xi_{\al},\eta_{\al})$ $(\al=1,2,3)$ consisting of endomorphisms
$\f_{\al}$ of the tangent bundle,  Reeb vector fields $\xi_{\al}$ and their dual contact 1-forms
$\eta_{\al}$ satisfying the following identities:
\begin{equation}\label{str}
\begin{array}{c}
\f_{\al}\circ\f_{\bt} = -\delta_{\al\bt}I + \xi_{\al}\otimes\eta_{\bt}+\epsilon_{\al\bt\gm}\f_{\gm},\\
\f_{\al}\xi_{\bt} = \epsilon_{\al\bt\gm}\xi_{\gm},\quad
\eta_{\al}\circ\f_{\bt}=\epsilon_{\al\bt\gm}\eta_{\gm},\quad \eta_{\al}(\xi_{\bt})=\delta_{\al\bt},
\end{array}
\end{equation}
where $\al,\bt,\gm\in\{1,2,3\}$, $I$ is the identity on the algebra $\X(\MM)$
, $\delta_{\al\bt}$ is the Kronecker delta, $\epsilon_{\al\bt\gm}$ is the Levi-Civita symbol, i.e. either the sign of the permutation $(\al,\bt,\gm)$ of $(1,2,3)$ or 0 if any index is repeated.




In this section we discuss the coherence of compatible metrics and B-metrics in an almost contact 3-structure.
In \cite{Kuo}, it is considered the case of a Riemannian metric which is compatible by equations \eqref{A-met} for the three almost contact structures.

Suppose that $\MM$ admits two almost contact structures $(\f_{\al},\xi_{\al},\eta_{\al})$, $(\al=2,3)$. If a pseudo-Riemannian metric $g$ is a B-metric for the both structures, then the property in the first line of \eqref{str} implies the properties in the second line of the same equations.

In \cite{Kuo} for the case of Riemannian metrics (positive definite), it is proved that if the almost contact 3-structure admits two almost contact metric structures, then the third one is of the same type.
We consider the relevant cases for our investigations in the following theorem.

\begin{thm}\label{thm:3str}
Let $\MM$ admit an almost contact 3-structure $(\f_{\al},\xi_{\al},\eta_{\al}),$ $(\al=1,\allowbreak{}2,\allowbreak{}3)$
and a pseudo-Riemannian metric $g$.
If one of the three structures $(\f_{\al},\xi_{\al},\eta_{\al},g)$ 
is an almost contact B-metric structure, then other two ones are an almost contact metric structure and an almost contact B-metric structure.
\end{thm}
\begin{proof}
First we establish on $\MM$ that if the pseudo-Riemannian metric $g$ and two of the almost contact structures generate:
\begin{enumerate}
  \item two almost contact metric structures, then the third one is an almost contact metric structure;
  \item two almost contact B-metric structures, then the third one is an almost contact metric structure;
  \item an almost contact metric structure and an almost contact B-metric structure, then the third one is an almost contact B-metric structure.
\end{enumerate}

Now, we argue for the case (ii). Let $(\f_{2},\xi_{2},\allowbreak{}\eta_{2},g)$ be an almost contact B-metric structure, i.e. \eqref{B-met} holds. Moreover, $(\f_{3},\xi_{3},\eta_{3},g)$ be also an almost contact B-metric structure, i.e. the following properties are valid
\begin{equation}\label{B-met3}
\begin{array}{c}
  g(\xi_{3},\xi_{3})=1, \quad \eta_{3}(x)=g(\xi_{3},x),\\
  g(\f_{3} x,\f_{3} y)=-g(x,y)+\eta_{3}(x)\eta_{3}(y).
\end{array}
\end{equation}
Then, by virtue of $\f_1=\f_2\circ\f_3-\xi_2\otimes\eta_3$, $\eta_2\circ\f_2=0$ and $\eta_2\circ\f_3=\eta_1$, which are consequences of \eqref{str}, and using \eqref{B-met} and \eqref{B-met3}, we obtain
\[
\begin{array}{l}
g(\f_1x,\f_1y)=g\bigl(\f_2(\f_3x)-\eta_3(x)\xi_2,\f_2(\f_3y)-\eta_3(y)\xi_2\bigr)\\
\phantom{g(\f_1x,\f_1y)}
=g(x,y)+\eta_1(x)\eta_1(y).
\end{array}
\]
Therefore, comparing with \eqref{A-met}, the metric $g$ is a compatible metric with respect to the almost contact structure $(\f_{1},\xi_{1},\eta_{1})$.

The verifications of the other cases are similar.
\end{proof}

Since any compatible metric and any B-metric on an almost contact manifold $\MM$ are metrics corresponding to a Hermitian metric and a Norden metric on the corresponding almost complex manifold $\MM\times\R$ (or on the corresponding contact distribution $H=\ker(\eta)$), respectively, we said that the compatible metric and the B-metric are metrics of Hermitian type and Norden type on $\MM$, respectively. Then, we give the following
\begin{defn}
We call a pseudo-Riemannian metric $g$ a \emph{metric of Hermitian-Norden type} (in short an \emph{HN-metric}) on a manifold with almost contact 3-structure $(\MM,\allowbreak{}\f_{\al},\allowbreak{}\xi_{\al},\allowbreak{}\eta_{\al})$, if it satisfies the identities
\begin{equation}\label{HN-met}
  g(\f_{\al}x,\f_{\al}y)=\ea g(x,y)+\eta_{\al}(x)\eta_{\al}(y),\quad \al=1,2,3
\end{equation}
for some cyclic permutation $(\ep_1,\ep_2,\ep_3)$ of $(1,-1,-1)$.
Then, $(\f_{\al},\xi_{\al},\allowbreak{}\eta_{\al},g)$ we call an \emph{almost contact HN-metric 3-structure}.

Let us suppose for the sake of definiteness that
\begin{equation}\label{ea}%
 \ea=
\begin{cases}
\begin{array}{ll}
1, \quad & \al=1;\\
-1, \quad & \al=2,3.
\end{array}
\end{cases}
\end{equation}
\end{defn}

As a sequel of \eqref{HN-met} we have the following properties for all $\al=1,2,3$:
\begin{equation}\label{eta-g}
\eta_{\al}=-\ea \xi_{\al} \lrcorner\, g,
\end{equation}
\begin{equation}\label{HN-met-2f}
  g(\f_{\al}x,y)=-\ea g(x,\f_{\al}y).
\end{equation}

Bearing in mind \eqref{HN-met-2f}, we deduce the following.
In the case $\al=1$, the associated tensor field of type $(0,2)$ is a 2-form. Let us denote it by $\widetilde g_{1}$, i.e. $\widetilde g_{1}(x,y)=g(\f_{1}x,y)$. It is actually opposite to $\Phi(x,y)=g(x,\f_{1}y)$, known as the \emph{fundamental} 
\emph{2-form} of the almost contact metric structure.
In other two cases $\al=2$ and $\al=3$, the tensor $(0,2)$-field $g(\f_{\al}x,y)$ is symmetric as well as $\etaa\otimes\etaa$. Then, we define the following fundamental tensor $(0,2)$-fields by
\begin{equation}\label{HN-met2}
  \widetilde g_{\al}(x,y)=g(\f_{\al}x,y)+\eta_{\al}(x)\eta_{\al}(y),\quad \al=2,3.
\end{equation}
Therefore, $\widetilde g_{2}$ and $\widetilde g_{3}$ are also HN-metrics, which we call \emph{associated metrics} to $g$ with respect to $(\f_{\al},\xi_{\al},\eta_{\al})$ for $\al=2$ and $\al=3$, respectively.

Bearing in mind the structural groups of the almost contact 3-structures with compatible metric (\cite{Kuo}) and the hypercomplex manifolds with HN-metric structure (\cite{GriMan24}), we can conclude the following.
The structural group of the manifolds with almost contact HN-metric 3-structure is $(GL(n,\mathbb{H})\cap O(2n,2n))\times O{(2,1)}$, where $GL(n,\mathbb{H})$ is  the group of invertible quaternionic $(n\times n)$-matrices and $O(p,q)$ is the pseudo-orthogonal group of signature $(p,q)$ for natural numbers $p$ and $q$.

The fundamental tensors of a manifold with almost contact HN-metric 3-struc\-ture are the three
$(0,3)$-tensors determined by
\begin{equation}\label{F}
F_\al (x,y,z)=g\bigl( \left( \n_x \f_\al
\right)y,z\bigr),\qquad
\al=1,2,3,
\end{equation}
where $\n$ is the Levi-Civita connection generated by $g$.
They have the following basic properties caused by the structures
\begin{equation}\label{Fa-prop}
\begin{array}{l}
  F_{\al}(x,y,z)=-\ea F_{\al}(x,z,y)\\
  \phantom{F_{\al}(x,y,z)}
  =-\ea F_{\al}(x,\f_{\al}y,\f_{\al}z)+F_{\al}(x,\xi_{\al},z)\,
  \eta_{\al}(y)\\
  \phantom{F_{\al}(x,y,z)=-\ea F_{\al}(x,\f_{\al}y,\f_{\al}z)}
    +F_{\al}(x,y,\xi_{\al})\,\eta_{\al}(z).
\end{array}
\end{equation}

The following associated 1-forms, defined as traces of $F_{\al}$, are known as their \emph{Lee forms}:
\begin{equation}\label{ta}
\begin{array}{c}
\theta_{\al}(z)=g^{ij}F_{\al}(e_i,e_j,z),\quad \theta^*_{\al}(z)=g^{ij}F_{\al}(e_i,\f_{\al}e_j,z),\\ \om_{\al}(z)=F_{\al}(\xi_{\al},\xi_{\al},z),
\end{array}
\end{equation}
where $g^{ij}$ are the components of the inverse matrix of the metric $g$ with respect to an arbitrary basis of the type $\{e_1,e_2,\dots,e_{4n+2},\xi_{\al}\}$.

The simplest case  of the manifolds with
almost contact HN-metric 3-structure is when the
structures are $\n$-parallel, $\n\f_\al=\n\xi_\al=\n\eta_\al=\n g=\n
\widetilde{g}_\al=0$, and it is determined by the condition
$F_\al=0$.
We call these manifolds \emph{cosymplectic HN-metric manifolds} or \emph{manifolds with cosymplectic HN-metric 3-structure}.


\section{Relation with pseudo-Riemannian manifolds equipped with almost complex or almost hypercomplex structures}

We can consider each of the three $(4n+2)$-dimensional  distributions
$H_{\al}=\ker(\eta_{\al})$, $\al\in\{1,2,3\}$, equipped with a corresponding pair of an almost complex structure
$J_{\al}=\f_{\al}|_{H_{\al}}$ and a metric $h_{\al}=g|_{H_{\al}}$, where $\f_{\al}|_{H_{\al}}$, $g|_{H_{\al}}$ are the restrictions of $\f_{\al}$, $g$ on $H_{\al}$, respectively, and the metrics $h_{\al}$
are compatible with $J_{\al}$ as follows
\begin{equation}\label{norden}
\begin{array}{l}
h_{\al}(J_{\al}X,J_{\al}Y)=\ea h_{\al}(X,Y) ,\quad \\
\widetilde{h}_{\al}(X,Y):=h_{\al}(J_{\al}X,Y)=-\ea h_{\al}(X,J_{\al}Y).
\end{array}
\end{equation}
Obviously, in the cases $\al=2$ and $\al=3$ the metrics $h_{\al}$ and their associated $(0,2)$-ten\-sors $\widetilde{h}_{\al}$ are Norden metrics, whereas for $\al=1$ the structure $(J_1,h_{1})$ is an almost Hermitian pseudo-Riemannian structure with a fundamental 2-form $\Omega=-\widetilde{h}_{1}$.
In such a way, any of the distributions $H_{\al}$ for $\al=2$ or $\al=3$ can be considered as a $(2n+1)$-dimensional
complex Riemannian distribution with a complex Riemannian metric
$g_{\al}^{\mathbb{C}}=h_{\al}+\widetilde h_{\al}\sqrt{-1}=g|_{H_{\al}}+\widetilde{g}|_{H_{\al}}\sqrt{-1}$.
In another point of view, the distribution $H_{\al}$ for $\al=2$ or $\al=3$ is a $(4n+2)$-dimensional almost complex distribution with a Norden metric $h_{\al}$ and its associated Norden metric $\widetilde h_{\al}$.
Moreover, the $4n$-dimensional distribution $H=H_1\cap H_2\cap H_3$ has an almost hypercomplex structure $(J_1,J_2,J_3)$, i.e. $J_{\al}^2=-I$, $J_3=J_1J_2=-J_2J_1$, $J_{\al}=\f_{\al}\vert_{H}$,
with a pseudo-Riemannian metric $h=g|_H$ which is Hermitian with respect to $J_1$ and a Norden metric with respect to $J_2$ and $J_3$  since
$
h(J_{\al}X,J_{\al}Y)=\ea h(X,Y)$, $\al=1,2,3$.

%

Let the vector $4n$-tuple
\[
(e_1, \dots, e_n; J_1e_1, \dots, J_1e_n; J_2e_1, \dots, J_2e_n; J_3e_1,\allowbreak{} \dots,\allowbreak{} J_3e_n)
\]
be an adapt\-ed basis  (or a $J_{\al}$-basis) of the almost hypercomplex structure. 
Then, according to \eqref{HN-met}, the basis
\begin{equation}\label{f-baza}
(e_1, \dots , e_n; \f_1e_1, \dots , \f_1e_n; \f_2e_1, \dots , \f_2e_n; \f_3e_1, \dots , \f_3e_n;\xi_1,\xi_2,\xi_3)
\end{equation}
is an an \emph{adapted basis} (or a $\f_{\al}$-basis) for the almost contact 3-structure and it is orthonormal with respect to $g$, i.e.
\begin{equation}\label{g-basis}
\begin{array}{l}
  g(e_i,e_i)=\ea g(\f_{\al}e_i,\f_{\al}e_i)=-\ea g(\xi_{\al},\xi_{\al})=1,\\
  g(e_i,e_j)=g(e_i,\f_{\al}e_i)=g(e_i,\f_{\al}e_j)=g(e_i,\xi_{\al})=g(\f_{\bt}e_i,\xi_{\al})=0
\end{array}
\end{equation}
for arbitrary $i\neq j\in\{1,2,\dots,n\}$ and $\al,\bt=1,2,3$.

It is well known that an even-dimensional almost complex manifold
$(N,J,h)$ endowed with a compatible Riemannian metric, i.e. $h(JX,JY)=h(X,Y)$, is an almost Hermitian manifold. There are considered also almost pseudo-Hermitian manifolds, i.e. the case when $h$ is a pseudo-Riemannian metric with the same compatibility (cf. \cite{Mats81}). 
We recall that
$(N,J,h)$ equipped with a pseudo-Riemannian metric of neutral signature
satisfying the identity $h(JX,JY)=- h(X,Y)$ is known as an almost complex
manifold with Norden metric \cite{N1,GGM85}, 
a generalized B-manifold \cite{GriMekDje85a}, 
an almost
complex manifold with B-metric \cite{GGM87} 
or an almost complex
manifold with complex Riemannian metric \cite{LeB,Manin,GaIv}. 
In the case when the almost complex structure $J$ is parallel with respect to
the Levi-Civita connection $\nabla^h$ of the metric $h$,
i.e. $\nabla^hJ=0$, then the manifold is known as a K\"ahler-Norden
manifold, a K\"ahler manifold with B-metric or a holomorphic
complex Riemannian manifold. Then the almost complex
structure $J$ is integrable and the local components of the complex metric
in holomorphic coordinate system are holomorphic functions.

From another point of view, it is well known the notion of the almost hypercomplex manifold with Hermitian metric (cf. \cite{AlMa}). The case when the metric is pseudo-Riemannian of neutral signature where one of the almost complex structures acts as an anti-isometry is considered in \cite{GriManDim12,GriMan24,ManGri32}. Then the almost hypercomplex structure $(J_1,J_2,J_3)$ and the metric $h$ generate two almost complex structures with Norden metrics (e.g., for $\al=2,3$) and one almost complex structure with Hermitian pseudo-Riemannian metric (e.g., for $\al=1$) because of \eqref{norden}.


The manifolds with almost contact HN-metric 3-structure can be considered as real hypersurfaces of an almost hypercomplex HN-metric manifold.
In \cite{GriManDim12} it is proved that every almost hypercomplex HN-metric manifold with parallel almost complex structures is flat.


In case of cosymplectic HN-metric manifolds the distribution $H$ is involutive. The
corresponding integral submanifold is a totally geodesic
submanifold which inherits a holomorphic hypercomplex Riemannian
structure and the manifolds with almost contact HN-metric 3-structure is
locally  a pseudo-Riemannian product of a holomorphic hypercomplex
Riemannian manifold with a 3-dimensional Lorentzian real space.

\section{Curvature properties of manifolds with almost contact HN-metric 3-structure}

A tensor $L$ of type $(0,4)$ with the prop\-er\-ties:%
\begin{equation}\label{curv}%
\begin{array}{l}%
L(x,y,z,w)=-L(y,x,z,w)=-L(x,y,w,z),\\
L(x,y,z,w)+L(y,z,x,w)+L(z,x,y,w)=0
\end{array}%
\end{equation} %
is called a \emph{curvature-like tensor}.

We say that a curvature-like tensor $L$ is a \emph{K\"ahler-like
tensor} on a manifold with almost contact HN-metric 3-structure when $L$ satisfies the properties: %
\begin{equation}\label{L-kel}%
L(x,y,z,w)=\ea L(x,y,\f_{\al} z,\f_{\al} w).
\end{equation} 
Using \eqref{str} and \eqref{HN-met}, we obtain that for a K\"ahler-like tensor $L$ the following properties are valid
\begin{equation}\label{L-kel2}%
\begin{array}{c}
L(x,y,z,w)=\ea L(x,\f_{\al} y,\f_{\al} z,w)=\ea L(\f_{\al} x,\f_{\al} y,z,w)\\
L(\xi_{\al},y,z,w)=L(x,\xi_{\al},z,w)=L(x,y,\xi_{\al},w)=L(x,y,z,\xi_{\al})=0.
\end{array}
\end{equation}
These properties show that if $L$ is a K\"ahler-like tensor on a manifold with almost contact HN-metric 3-structure then $L$ is a K\"ahler-like tensor on $(H,J_{\al}=\f_{\al}|_H,h=g|_H)$ which is a manifold with almost hypercomplex HN-metric structure. It is known from \cite{Man28} that every K\"ahler-like tensor vanishes on an almost hypercomplex HN-metric manifold. Therefore, it is valid the following
\begin{prop}\label{prop:L=0}
Every K\"ahler-like tensor vanishes on a manifold with almost contact HN-metric 3-structure.
\end{prop}

Let $R$ be the curvature tensor of the Levi-Civita connection
$\n$ generated by $g$.
It is defined as usual by
$R(x,y)z=\left[\n_x, \n_y\right] z - \n_{\left[x,y\right]} z$. The corresponding $(0,4)$-tensor,
denoted by the same letter, is determined by
$R(x,y,z,w)=g\left(R(x,y)z,w\right)$.

According to \cite{GriManDim12}, every hyper-K\"ahler HN-metric manifold is flat.
Since $R$ is a K\"ahler-like tensor on every manifold with cosymplectic HN-metric 3-structure, i.e. $\n\f_\al=0$ for all $\al=1,2,3$, then applying \propref{prop:L=0} we obtain
\begin{prop}\label{prop:R=0}
Every manifold with cosymplectic HN-metric 3-structure is flat.
\end{prop}

%

\section{Examples of manifolds with almost contact HN-metric 3-structure}

\subsection{A real vector
space with contact HN-metric 3-structure}

Let $\V$ be a real vector space of dimension $(4n+3)$ and a (local) basis of $\V$ is denoted by
$
\left\{\ddx;\ddy;\ddu;\ddv;\dda,\right.$ $\left.\ddb,\ddc\right\}$, $i=1,2,\dots,n.
$
%
%
Every vector $z$ of $\V$ is represented in
the mentioned basis by its coordinates $z(x^i;y^i;u^i;v^i;\allowbreak{}a,\allowbreak{}b,c)$. 
%
%
Then, we consider
\begin{equation}\label{x-coord}
\begin{array}{lll}
\f_1z(-y^i;x^i;v^i;-u^i;0,-c,b),\quad & \eta_1(z)=a,\quad & \xi_1=\dda,\\
\f_2z(-u^i;-v^i;x^i;y^i;c,0,-a),\quad & \eta_2(z)=b,\quad & \xi_2=\ddb,\\
\f_3z(v^i;-u^i;y^i;-x^i;-b,a,0),\quad & \eta_3(z)=c,\quad & \xi_3=\ddc.
\end{array}
\end{equation}

\begin{defn}\label{d1}
A structure $(\f_{\al},\xi_{\al},\eta_{\al})$, $\al=1,2,3$ on
$\V$ is called \emph{a contact 3-structure} on
$\V$.
\end{defn}

Let us introduce a pseudo-Euclidian metric $g$ of
signature $(2n+2,2n+1)$ as follows
\begin{equation}\label{metric}
g(z,z')=\sum_{i=1}^n \left(x^ix'^i+y^i y'^i-u^iu'^i-v^iv'^i\right)-aa'+bb'+c c',
\end{equation}
where $z(x^i;y^i;u^i;v^i;a,b,c)$, $z'(x'^i;y'^i;u'^i;v'^i;a',b',c') \in \V$, $i=1,2,\dots,n$. This metric satisfies the following properties
\begin{equation}\label{metr-sv}
\begin{array}{c}
g(\f_{\al}z,\f_{\al}z')=\ea g(z,z') +\eta_{\al}(z)\eta_{\al}(z').
\end{array}
\end{equation}

Let $\n$ be the Levi-Civita connection of $g$. We check immediately that $\n\f_{\al}=0$ and therefore we get the following
\begin{prop}
The space $(\V,\f_{\al},\xi_{\al},\eta_{\al},g)$ is 
a manifold with cosymplectic HN-metric 3-structure.
\end{prop}

\subsection{A time-like sphere with almost contact HN-metric 3-structure}

It is known that any real hypersurface of an almost hypercomplex manifold carries in a natural way an almost contact 3-structure.

In a similar way it can be shown that on every real nonisotropic hypersurface of an almost hypercomplex HN-metric manifold there arises an almost contact HN-metric 3-structure.

Let $\R^{4n+4}=\bigl\{\left(x^i;y^i;u^i;v^i\right)\vert\ x^i,y^i,u^i,v^i\in\R,\ i\in\{1,2,\dots,n+1\}\bigr\}$ be a vector space of dimension $4n+4$ with an almost hypercomplex structure $(J_{1},J_{2},J_{3})$ determined as follows \cite{GriManDim12}
\[
J_1z(-y^i;x^i;v^i;-u^i),\quad
J_2z(-u^i;-v^i;x^i;y^i),\quad
J_3z(v^i;-u^i;y^i;-x^i)
\]
for an arbitrary vector $z(x^i;y^i;u^i;v^i)$.
This space is equipped with a pseudo-Eu\-clidean metric of neutral signature, i.e. $(2n+2,2n+2)$, by
\[
g(z,z')=\sum_{i=1}^{n+1} \bigl(x^ix'^i+y^iy'^i-u^iu'^i-v^iy'^i\bigr)
\]
for arbitrary $z(x^i;y^i;u^i;v^i), z'(x'^i;y'^i;u'^i;v'^i)\in\R^{4n+4}$.

Identifying an arbitrary point $p\in\R^{4n+4}$ with its position vector $z$, we study the following hypersurface of
$\R^{4n+4}$.

Let $\SSS: g(z,z)=-1$ be the unit time-like sphere of $g$ in $\R^{4n+4}$. Then $z$ coincides with the unit normal $N$ to the tangent space $T_p\SSS$ at $p$.

For $\al=1,2,3$ we determine the Reeb vector fields by the equalities
$\xi_{\al}=\lm_{\al}\cdot N+\mu_{\al}\cdot J_{\al}N$,
such that $g(N,\xi_{\al})=0$ and $g(\xi_{\al},\xi_{\al})=-\ea$.
We substitute $g(N,J_{\al}N)=\tan\psi_{\al}$ for $\psi_{\al}\in\left(-\frac{\pi}{2},\frac{\pi}{2}\right)$. Then we obtain
\begin{equation*}
\xi_{\al}=\sin\psi_{\al}\cdot N+\cos\psi_{\al}\cdot J_{\al}N.
\end{equation*}
Since $g(N,J_{1}N)=0$ then $\psi_{1}=0$ and therefore $\xi_{1}=J_{1}N$.
Because of $g(N,\xi_{\al})=0$ we have that $\xi_{\al}$ are in $T_p\SSS$.
The conditions $g(\xi_{\al},\xi_{\bt})=0$ for $\al\neq\bt$ are equivalent to $\psi_{2}=\psi_{3}=0$.
 Therefore we obtain for all $\al$
\begin{equation}\label{xi-exa}
\xi_{\al}=J_{\al}N.
\end{equation}
Using the latter equality and $J_{\al}J_{\al}N=-N$, we obtain that $J_{\al}\xi_{\al}=-N$.

We define the structure endomorphisms $\f_{\al}$ ($\al=1,2,3$) and the contact 1-forms $\eta_{\al}$ in $T_p\SSS$ by the following orthonormal decomposition of $J_{\al}x$ for arbitrary $x\in T_p\SSS$
\begin{equation}\label{fi-eta-exa}
  J_{\al}x=\f_{\al}x-\eta_{\al}(x)\cdot N,
\end{equation}
i.e. $\f_{\al}x$ is the tangent component of $J_{\al}x$ and $-\eta_{\al}(x)N$ is the corresponding normal component.
By direct computation \eqref{fi-eta-exa} implies \eqref{str}. Then, using \eqref{xi-exa}, we obtain \eqref{HN-met} and \eqref{eta-g}.
Thus we equip the unit time-like sphere $\SSS$ in $\R^{4n+4}$ with an almost contact HN-metric 3-structure $(\f_{\al},\xi_{\al},\eta_{\al},g)$.

Let $\bar\n$ and $\n$ be the Levi-Civita connections of the metric $g$ in $\R^{4n+4}$ and $\SSS$, respectively. Since $\bar\n$ is flat, the formulas of Gauss and Weingarten have the form
\begin{equation}\label{GW}
\bar\n_xy=\n_xy+g(x,y)N,\qquad
\bar\n_xN=x.
\end{equation}
Therefore one can obtain by \eqref{xi-exa}, \eqref{fi-eta-exa} and \eqref{GW} that
\[
\n_x\xi_{\al}=\f_{\al}x,
\qquad
F_{\al}(x,y,\xi_{\al})=-g(\f_{\al}x,\f_{\al}y).
\]
Then, for the Lee forms we have
\[
\ta_{\al}(\xi_{\al})=4n+2,\qquad \ta^*_{\al}(\xi_{\al})=0,\qquad \om_\al=0.
\]
Finally, we get
\begin{equation}\label{F-S}
F_{\al}(x,y,z)=-
g(\f_{\al}x,\f_{\al}y)\eta_{\al}(z)
+\ea g(\f_{\al}x,\f_{\al}z)\eta_{\al}(y).
\end{equation}

In the case $\al=1$, the equality \eqref{F-S} takes the form
\[
F_{1}(x,y,z)=
-g(x,y)\eta_{1}(z)
+g(x,z)\eta_{1}(y),
\]
i.e. by virtue of \eqref{HN-met}, \eqref{ea}, \eqref{eta-g} and \eqref{F}, we have
\[
(\n_x \f_{1})y = g(x,y)\xi_{1}-g(\xi_{1},y)x.
\]
According to \cite[Theorem 6.3]{Blair}, the latter equality is a necessary and sufficient condition for a Sasakian manifold.

Similarly, in the case $\al=2$ or $\al=3$,
from \eqref{F-S}, according to \cite[The\-o\-rem 3.3]{IvManMan45}, we get a necessary and sufficient condition for a Sasaki-like almost contact complex Riemannian manifold.

We recall that a Sasakian manifold (respectively, a Sasaki-like almost contact complex Riemannian manifold) is defined as an almost contact metric manifold (respectively, an almost contact B-metric manifold) which complex cone is a K\"ahler manifold  (respectively, a K\"ahler-Norden manifold) (cf. \cite{Blair}, \cite{IvManMan45}).

Thus, we obtain the following
\begin{prop}
The manifold $(\SSS,\f_{\al},\xi_{\al},\eta_{\al},g)$ is$:$
  \begin{enumerate}
    \item a Sasakian manifold for $\al=1;$
    \item a Sasaki-like almost contact complex Riemannian manifold for $\al=2,3.$
\end{enumerate}
\end{prop}

In \cite{GaMiGr}, it is considered a unit time-like sphere with almost contact B-met\-ric structure and it is proved that it belongs to the class $\F_4\oplus\F_5$, the analogue of trans-Sasakian manifold of type $(\al,\bt)$.

%




\end{document}